\documentclass[reqno]{amsart}
\usepackage{graphicx,amssymb,cite}
\usepackage[colorlinks,plainpages, citecolor=magenta, linkcolor=blue, bookmarksnumbered]{hyperref}

\newtheorem{thm}{Theorem}

\newtheorem{lem}{Lemma}

\theoremstyle{definition}

\newtheorem{exm}{Example}
\theoremstyle{remark}

\numberwithin{equation}{section}

\newcommand{\dmn}{\mathop{\rm dom}}

\newcommand{\supp}{\mathop{\rm supp}}

\renewcommand{\kappa}{\varkappa}

\newcommand{\Real}{\mathbb R}

\newcommand{\eps}{\varepsilon}

\newcommand{\cH}{\mathcal{H}}

\newcommand{\cV}{\mathcal{V}}
\newcommand{\cU}{\mathcal{U}}

\newcommand{\cD}{\mathcal{D}}

\newcommand{\cR}{\mathcal{R}}

\newcommand{\cS}{\mathcal{S}}

\newcommand\xe{\left(\tfrac x\eps\right)}

\newcommand{\ra}{\rangle}
\newcommand{\la}{\langle}

\begin{document}
\title[On negative eigenvalues of 1D Schr\"{o}dinger operators]{On negative eigenvalues of 1D Schr\"{o}dinger operators with $\delta'$-like potentials}

\author{Yuriy Golovaty} \address[YG]{Ivan Franko National University of Lviv,
	1 Universytetska st., 79602 Lviv, Ukraine}%
\email{yuriy.golovaty@lnu.edu.ua}
\author{Rostyslav Hryniv} \address[RH]{
	Ukrainian Catholic University, 2a Kozelnytska str., 79026, Lviv, Ukraine \and
	University of Rzesz\'{o}w, 1 Pigonia str., 35-310 Rzesz\'{o}w, Poland}%
\email{rhryniv@ucu.edu.ua, rhryniv@ur.edu.pl}

\subjclass[2020]{Primary ; Secondary }

\begin{abstract}
We study the asymptotic behavior of the discrete spectrum of one-dimensional Schr\"{o}dinger operators with $\delta'$-like potentials, which are used to construct exactly solvable models for localized dipoles in quantum mechanics. Although these operators converge in the norm resolvent topology to a limiting operator that is bounded from below, we prove that they can possess a finite but arbitrarily large number of discrete eigenvalues that diverge to negative infinity as the regularization parameter tends to zero. This phenomenon illustrates a spectral instability of Schr\"odinger operators with these singular potentials.
\end{abstract}
\maketitle

\section{Introduction}

The aim of this paper is to establish the existence and describe the asymptotic behavior of negative eigenvalues of one-dimensional Schrödinger operators that serve as regularizations of formal Hamiltonians involving $\delta$ and $\delta'$ potentials. These questions arise in the construction and analysis of exactly solvable models in quantum mechanics, a topic that continues to draw considerable attention in the literature (see~\cite{Albeverio2edition, AlbeverioKurasov}, as well as comprehensive reference lists therein, covering works up to the early 2000s).

Some point interactions (i.e., pseudopotentials supported on discrete sets) naturally lead to well-defined exactly solvable models; others, however, exhibit essential ambiguities in defining the corresponding Hamiltonians. A notable example of this contrast is provided by the $\delta$ and $\delta'$ potentials. In the case of the $\delta$ potential, the differential equation  $-y''+\alpha\delta(x)y=\lambda y$ is well-posed in the space of distributions $\cD'(\Real)$ and has a two-dimensional solution space. In contrast, the equation $-y''+\alpha\delta'(x)y=\lambda y$ is ill-posed in $\cD'(\Real)$ and admits only the trivial solution when $\alpha\neq 0$. Moreover, while every reasonable regularization of Hamiltonians involving the $\delta$ potential yields the same exactly solvable model, the $\delta'$ potential is sensitive to the regularization procedure, and different approximations may lead to different point interactions. As a result, the choice of the exactly solvable model for $\delta'$ potentials is not determined by mathematical considerations alone. However, it must reflect the specifics of the particular physical experiment---a feature that only enriches the study of exactly solvable models.

In this paper, we demonstrate a further distinction between the $\delta$ and $\delta'$ potentials, this time concerning the spectral properties of their regularized Hamiltonians. Natural approximations of both the $\delta$ and $\delta'$ potentials by regular potentials yield operator families that converge in the norm resolvent topology to semi-bounded limits. However, we show that in contrast to $\delta$-like perturbations, $\delta'$-like perturbations lead to operator families that are not uniformly bounded from below as the regularization parameter tends to zero. As a consequence, such regularized Hamiltonians can possess a finite (but arbitrarily large) number of low-lying eigenvalues that diverge to negative infinity. We explicitly determine the number of these eigenvalues and describe their asymptotic behavior in the singular limit.

The rest of the paper is organized as follows.  Section~\ref{sec:review} gives a brief overview of studies on exactly solvable models for Hamiltonians with $\delta$ and $\delta'$ potentials. In Section~\ref{sec:existence}, we derive conditions under which Schr\"odinger operators with $\delta'$-like potentials possess low-lying eigenvalues that diverge to negative infinity as the regularization parameter tends to zero.
Section~\ref{sec:counting} introduces methods for estimating the number of such eigenvalues and shows, in particular, that the emergence of a discrete spectrum is closely related to zero-energy resonances for the corresponding Schr\"odinger operators. In Section~\ref{sec:deltagenerated} we investigate how the non-trivial interaction of $\delta$-like and $\delta'$-like perturbations leads to the emergence of a negative eigenvalue with a finite limit as the perturbation parameter tends to zero. Finally, Section~\ref{sec:asymptotics} contains the proofs of Theorems~\ref{TheoremDeltaPrime} and~\ref{TheoremDelta} on the asymptotic behavior of eigenvalues.


\section{Short review of exactly solvable models for $\delta$ and $\delta'$ potentials}\label{sec:review}


In this section, we review existing approaches to constructing exactly solvable quantum mechanical models for one-dimensional Hamiltonians with pseudo-potentials involving the Dirac $\delta$-function and its derivative $\delta'$.

The simplest case is the formal (pseudo-)Hamiltonian
\begin{equation}\label{pseudoHDelta}
	- \frac{d^2}{dx^2} + \alpha \delta(x),\quad \alpha\in\Real.
\end{equation}
Any reasonable method of associating a self-adjoint Hamiltonian to~\eqref{pseudoHDelta}---such as form-sum, generalized sum method, appro\-xi\-mation by regular potentials---yields the same operator $H$, acting as $H y = - y''$ on the domain
\begin{equation*}
	\dmn H = \big\{y\in W^2_2(\Real\setminus{0}) \colon y(+0)=y(-0),\;\; y'(+0) - y'(-0) = \alpha y(0)\big\}.
\end{equation*}
In other words, the distributional potential $\alpha\delta(x)$ in the one-dimensional Schr\"odinger operator results in the point interaction imposing the interface condition
\begin{equation}\label{DeltaPointInteraction}
\begin{pmatrix}
  y(+0)\\y'(+0)
\end{pmatrix}=
\begin{pmatrix}
  1& 0\\
  \alpha & 1
\end{pmatrix}
\begin{pmatrix}
  y(-0)\\y'(-0)
\end{pmatrix}.
\end{equation}
Moreover, this model serves as a good approximation in the norm resolvent sense of the Schr\"odinger operators with integrable potentials of special form. Specifically, given a real-valued function $U$ of compact support such that
\begin{equation*}
    \alpha=\int_{\Real} U(x)\,dx,
\end{equation*}
the scaled potentials $U_\eps(x) = \eps^{-1} U(\eps^{-1}x)$ converge in the space of distributions $\cD'$ to the distribution $\alpha\delta(x)$ as $\eps\to 0$, and the corresponding operators
\begin{equation*}
  - \frac{d^2}{dx^2} + \eps^{-1} U(\eps^{-1}x)
\end{equation*}
converge to $H$ in the norm resolvent sense~\cite[Theorem~I.3.2.3]{Albeverio2edition}, i.e. their resolvents converge in operator norm to the resolvent of $H$. Similar convergence results hold even in the presence of background potentials $W$, i.e., for operators of the form
\begin{equation*}
	- \frac{d^2}{dx^2} + W(x) + \alpha \delta(x).
\end{equation*}


One should not expect that every pseudopotential gives rise to a unique point interaction. Certain pseudopotentials are highly sensitive to the way they are approximated, and the $\delta'$ potential is one of them. In physics, the symbol $\delta'$ is often used to describe a strongly localized dipole-type potential, such as a high narrow barrier followed by a deep well. Let $V$ be an integrable function with compact support and a finite first moment; then the sequence $\eps^{-2} V(\eps^{-1}x)$ converges in $\cD'$, as $\eps\to0$, if and only if $\int_\Real V(x) \,dx=0$, and in that case
\begin{equation}\label{DeltaPrimeConv}
  \eps^{-2} V(\eps^{-1}x)\to \beta \delta'(x),
\end{equation}
where $\beta=-\int_\Real x V(x)\,dx$. For this reason, we refer to such families of scaled potentials as $\delta'$-\textit{like}.

The question of how to correctly define the formal Hamiltonian
\begin{equation}\label{eq:BetaDeltaPrime}
	- \frac{d^2}{dx^2} + \beta \delta'(x)
\end{equation}
has a long and intricate history. As mentioned above, difficulties arise already at the level of interpreting the differential expression in~\eqref{eq:BetaDeltaPrime}, since the equation~$-y''+\beta\delta'(x)y=\lambda y$ admits only the trivial solution in the space of distributions~$\cD'$. Indeed, the product $\delta'(x)\phi(x)$ is well defined in $\cD'$ only if $\phi$ is continuously differentiable, and in that case it is equal to the distribution $\phi(0)\delta'(x) - \phi'(0)\delta(x)$. However, any nontrivial solution $y$ of the above equation would have to be discontinuous at the origin, since its second derivative $y'' = \beta\delta'(x)y - \lambda y$ would necessarily include a $\delta'$ term. In this case, the product $\delta'(x) y(x)$ is not defined in~$\cD'$, making the equation invalid.

Moreover, the operator in \eqref{eq:BetaDeltaPrime} cannot be rigorously defined using standard approaches like the form-sum or generalized sum methods, or as a relatively bounded perturbation of the free Hamiltonian. For this reason, it is natural to approach this problem via regularization: one considers  families of Schr\"odinger operators of the form
\begin{equation*}
	H_\eps= -\frac{d^2}{dx^2}+\eps^{-2}\,V(\eps^{-1}x), \qquad \dmn H_\eps=W^2_2(\Real),
\end{equation*}
with $\delta'$-like potentials $\eps^{-2}\,V(\eps^{-1}x)$ as a starting point for studying physical phenomena associated with zero-range dipoles. The construction of exactly solvable models for such dipole interactions is thus reduced to analyzing the limits of $H_\eps$ as $\eps\to0$.

It has been shown that the operator family $H_\eps$ indeed converges in the norm resolvent sense as $\eps \to 0$. In a seminal paper, \v{S}eba~\cite{SebRMP} argued that the limiting operator is the free Hamiltonian~$D_0$ decoupled at the origin by the Dirichlet condition, namely
\begin{equation}\label{OprD0}
  D_0 y = - y'',\qquad \dmn D_0 = \big\{y\in W^2_2(\Real\setminus{0}) \colon y(0)= 0\big\}.
\end{equation}
According to this result, no meaningful definition of a Schr\"odinger operator with a $\delta'$-potential would be possible, since the limit $D_0$ is completely impenetrable to a quantum particle and is independent of the specific form of the function~$V$. However, this conclusion contradicts the findings of Zolotaryuk \textit{et al.}~\cite{ChristianZolotarIermak03, Zolotaryuk08, Zolotaryuk09, Zolotaryuk10}, who analyzed transmission probabilities through piecewise constant $\delta'$-like potentials and observed examples of quantum tunnelling. These results prompted the revision of~\cite{SebRMP}; it was later rigorously proved in~\cite{GolovatyHrynivJPA2010} that the operator $D_0$ is the norm resolvent limit of $H_\eps$ only in the so-called \textit{non-resonant} case, while in the \textit{resonant} case, the situation is different.

We begin by recalling the relevant definitions~\cite{Klaus1982,HintonKlausShaw91a,HintonKlausShaw91b}.
The operator~$-\frac{d^2}{dx^2} + V$ is said to have a \emph{zero-energy resonance} if the equation $v'' = Vv$ admits a non-trivial solution~$v$ that is bounded on the entire real line. Such a solution is called a \emph{half-bound state}, and the potential~$V$ is then referred to as \emph{resonant}. Every half-bound state~$v$ has finite, nonzero limits $v_\pm$ at $\pm\infty$, and the ratio
\begin{equation*}
    \theta = \frac{v_+}{v_-}
\end{equation*}
is uniquely determined by~$V$.
As proved in~\cite{GolovatyHrynivJPA2010} (see also \cite{GolovatyManko2009, GolovatyHrynivProcEdinburgh2013}), if the potential~$V$ is resonant, then the family $H_\eps$ converges in the norm resolvent sense as $\eps \to 0$ to the self-adjoint operator
\begin{equation*}
H(\theta)=-\frac{d^2}{dx^2},\;\; \dmn H(\theta) = \big\{y\in W^2_2(\Real\setminus{0} ): y(+0) = \theta y(-0),\; \theta y'(+0) = y'(-0)\}.
\end{equation*}
We call $H(\theta)$ the \textit{Schr\"odinger operator with $\delta'_\theta$ potential}, with $\theta$ specifying the above interface conditions.

Regardless of whether the family $\eps^{-2} V(\eps^{-1}x)$ converges in $\cD'$ as $\eps\to0$ or not, the Schr\"o\-din\-ger operators $H_\eps$ converge in the norm resolvent sense to either $H(\theta)$ or $D_0$ depending on whether $V$ is resonant or non-resonant. Moreover, there is no functional dependence between the constant $\beta$ appearing in the distributional limit~\eqref{DeltaPrimeConv} and the interface parameter~$\theta$ in the point interaction
\begin{equation}\label{ThetaInteractions}
\begin{pmatrix}
  y(+0)\\ y'(+0)
\end{pmatrix}=
\begin{pmatrix}
  \theta& 0\\
  0 & \theta^{-1}
\end{pmatrix}
\begin{pmatrix}
  y(-0)\\ y'(-0)
\end{pmatrix}
\end{equation}
corresponding to~$H(\theta)$.
Two different resonant, zero-mean potentials $V$ may produce the same $\beta$ but different values of $\theta$, and conversely, the same $\theta$ may arise for different $\beta$. It is worth noting that Kurasov~\cite{KurElaMSI1993, Kurasov1996} was the first to establish a connection between the $\delta'$-potential and the point interactions described by~\eqref{ThetaInteractions}.

In~\cite{GolovatyMFAT2012, GolovatyIEOT2012},
it was proved that the approximations of pseudo-Hamiltonians $ - \frac{d^2}{dx^2}+\alpha \delta(x)+ \beta \delta'(x)$ by the Hamiltonians
\begin{equation*}
   -\frac{d^2}{dx^2}+ \eps^{-1} U(\eps^{-1}x)+\eps^{-2}\,V(\eps^{-1}x)
\end{equation*}
with every regular functions~$U$ and $V$ also converge in the norm resolvent topology. If $V$ is non-resonant, the operators converge to $D_0$. However, if $V$ is reso\-nant with a half-bound state $v$, then the limiting operator is associated with point interaction producing the interface conditions
\begin{equation}\label{DD'interactions}
\begin{pmatrix}
    y(+0)\\y'(+0)
\end{pmatrix}=
\begin{pmatrix}
  \theta& 0\\
  \eta & \theta^{-1}
\end{pmatrix}
\begin{pmatrix}
    y(-0)\\y'(-0)
\end{pmatrix},
\end{equation}
where
\begin{equation}\label{ThetaEta}
 \theta = \frac{v_+}{v_-}, \qquad \eta=\frac1{v_-v_+}\int_{\Real} Uv^2\,dx.
\end{equation}
A comprehensive study of exactly solvable models with point interactions \eqref{DD'interactions} has been done by Gadella, Nieto  \textit{et al} \cite{GadellaHerasNietoas2009, GadellaNegroNieto2009, GadellaGlasserNieto2011}. Besides the approximation of pseudo\-potentials by regular potentials, there are other methods to construct exactly solvable models: e.g., the method of self-adjoint extensions has been used by Nizhnik \cite{NizhFAA2003, NizhFAA2006}, and the distributional approach has been proposed by Lunardi, Manzoni \textit{et al}\cite{LunardiManzoni2014, LunardiManzoni2016}.

We note that the point interactions characterized by the interface conditions
\begin{equation*}
\begin{pmatrix}
  y(+0)\\ y'(+0)
\end{pmatrix}=
\begin{pmatrix}
  1& \beta\\
  0 & 1
\end{pmatrix}
\begin{pmatrix}
  y(-0)\\ y'(-0)
\end{pmatrix},
\end{equation*}
commonly referred to as $\delta'$-\textit{interactions}, are also sometimes interpreted as models of the formal $\delta'$ potential.
Exner, Neidhardt, and Zagrebnov~\cite{ExnerNeidhardtZagrebnov2001} proposed a refined potential approximation of such interactions using a family of three $\delta$-like potentials centered at the points $\pm a$ and $0$, with the separation distance $a$ tending to zero in a carefully coordinated way with the coupling constants.
Further contributions in this direction include the works of Cheon and Shigehara~\cite{CheonShigehara1998}, Zolotaryuk~\cite{ZolotaryukThreeDelta17}, and Albeverio, Fassari, and Rinaldi~\cite{FassariRinaldi2009,AlbeverioFassariRinaldi2013,AlbeverioFassariRinaldi2016}; see also the recent publication~\cite{FassariGadellaNietoRinaldi2024}.
Although the potential families in~\cite{ExnerNeidhardtZagrebnov2001} do not converge to $\delta'$ in the sense of distributions, the term ``$\delta'$-interactions'' can be partially justified by interpreting  $\delta'$ as a finite-rank perturbation; see~\cite{AlbeverioKoshmanenkoKurasovNizhnik2002, KuzhelNizhnik2006} for further discussion.

Let $\la\,\cdot\,,\, \cdot\,\ra$ be the dual pairing between the Sobolev spaces $W_2^{-s}(\Real)$ and $W_2^{s}(\Real)$. Since $\delta(x)y(x)=y(0)\delta(x)=\la\delta, y\ra \delta(x)$,  the formal operator \eqref{pseudoHDelta} can be written as
\begin{equation}\label{DeltaRankOne}
  -\frac{d^2}{dx^2}+\alpha \la\delta,\, \cdot\,\ra \delta(x).
\end{equation}
This shows that the $\delta$-potential can be interpreted as a rank-one perturbation of the free Hamiltonian, and the standard theory of regular finite-rank perturbations yields the same exactly solvable model as in~\eqref{DeltaPointInteraction}. In the physical literature, the $\delta'$-interaction is typically associated with rank-one perturbation of the free Hamiltonian as in~\eqref{DeltaRankOne} but with $\delta'$ in place of $\delta$~\cite[Ch.~1.4]{Albeverio2edition}:
\begin{equation}\label{DeltaPrimRankOne}
  -\frac{d^2}{dx^2}+\beta \la\delta',\, \cdot\,\ra \delta'(x).
\end{equation}
The more general results of~\cite{AlbNizh2000, AlbeverioKoshmanenkoKurasovNizhnik2002} imply that there exist regular potentials $\phi_\eps, \psi_\eps \in L^2(\Real)$ converging to $\delta'$ in~$\cD'$ such that the rank-one perturbations of the free Hamiltonian,
\begin{equation*}
    -\frac{d^2}{dx^2}+ \beta  \la \phi_\eps,\ \cdot\ \ra \psi_\eps(x),
\end{equation*}
converge to~\eqref{DeltaPrimRankOne} in the strong resolvent topology as $\eps\to0$. The difference, in terms of spectral effects, between the perturbation of the so-called 1D conic oscillator consisting of a mixed potential $\alpha\delta + \beta\delta'$ and the one with the $\delta'$-interaction in \eqref{DeltaPrimRankOne} was investigated in detail in~\cite{Fassari2018}.

However, the model~\eqref{DeltaPrimRankOne} is not directly related to the formal expression \eqref{eq:BetaDeltaPrime} with a $\delta'$-potential. Indeed, if the product $\delta'(x)y(x)$ is well defined in the distributional sense, then
\[
    \delta'(x)y(x)=y(0)\delta'(x)-y'(0)\delta(x).
\]
Using this identity, the formal expression \eqref{eq:BetaDeltaPrime} can be interpreted as a rank-two perturbation of the free Schrödinger operator:
\begin{equation*}
  -\frac{d^2}{dx^2}+\beta \la\delta,\, \cdot\,\ra \delta'(x)+\beta \la\delta',\, \cdot\,\ra \delta(x).
\end{equation*}

In \cite{GolovatyJPA2018, GolovatyIEOT2018}, the norm resolvent convergence of the regular Hamiltonians
\begin{equation*}
  -\frac{d^2}{dx^2}+(g_\eps,\cdot) \,f_\eps
  +(f_\eps,\cdot) \,g_\eps +\eps^{-1}U\xe.
\end{equation*}
was studied. Here, $f_\eps$ and $g_\eps$ are sequences of complex-valued functions in $C_0^\infty(\Real)$ such that $f_\eps \to \delta'$ and $g_\eps \to \delta$ in the distributional sense, and $(\cdot, \cdot)$ denotes the inner product in $L^2(\Real)$. Under suitable assumptions on $f_\eps$, $g_\eps$, and the potential~$U$, such operators were shown to approximate the two-parameter family of point interactions defined by the interface conditions
\begin{equation*}
\begin{pmatrix}
  y(+0)\\ y'(+0)
\end{pmatrix}=
\begin{pmatrix}
  \mu & \kappa\\
  0 & \mu^{-1}
\end{pmatrix}
\begin{pmatrix}
  y(-0)\\ y'(-0)
\end{pmatrix}.
\end{equation*}

Although there is no established theory of distributions on metric graphs, the notions of $\delta$-like and $\delta'$-like potentials can be naturally extended to this setting. The construction of exactly solvable models on quantum graphs, as well as the approximation of singular vertex couplings---including mixed $\alpha\delta' + \beta\delta$ interactions---has been explored in~\cite{CheonExner2004, Manko2010, ExnerManko2013, GolovatyAHP2023}.

The above results illustrate the richness of approaches to modeling point interactions and exactly solvable models in quantum mechanics. While $\delta$-potentials admit a canonical interpretation, the situation becomes especially delicate when the formal $\delta'$-potential is involved, as different approximations may lead to different exactly solvable models. The choice of the appropriate limit operator is, therefore, not unique and must be guided by the physical or mathematical context of the problem.

%
\section{Existence of low-lying eigenvalues for $\delta'$-like potentials}\label{sec:existence}
%

Let us consider the family of operators
\begin{equation}\label{Heps}
 \cH_\eps= -\frac{d^2}{dx^2}+W(x)+ \eps^{-1} U(\eps^{-1}x)+\eps^{-2}\,V(\eps^{-1}x), \qquad \dmn \cH_\eps=W^2_2(\Real),
\end{equation}
where $U$, $V$ and $W$ are compactly supported $L^\infty(\Real)$-potentials. This restriction on the potentials avoids unnecessary technical complications; however, the results remain valid for a significantly broader class of potentials  (cf.~\cite{GolovatyHrynivProcEdinburgh2013} for an example of how this constraint can be relaxed). We are interested in the emergence of negative eigenvalues in Hamiltonians due to $\delta$-like and $\delta'$-like perturbations. Accordingly, we assume that $W \ge 0$, so that the unperturbed operator $H_0 = -\frac{d^2}{dx^2} + W$ is non-negative and has a purely continuous spectrum.

As follows from the result of \cite{GolovatyMFAT2012}, the operators $\cH_\eps$ converge in the norm resolvent sense as $\eps\to0$. If the potential $V$ is resonant, i.e., possesses a half-bound state $v$ (see Section~\ref{sec:review}), then $\cH_\eps$ converge to the operator
\begin{multline}\label{Hthetaeta}
  \cH=-\frac{d^2}{dx^2}+W, \quad \dmn \cH=\big\{\phi\in W^2_2(\Real\setminus\{0\})\colon \\
  \phi(+0)=\theta\phi(-0), \quad\phi'(+0)=\theta^{-1}\phi'(-0)+\eta \phi(-0)\big\},
\end{multline}
where $\theta$ and $\eta$ are given by \eqref{ThetaEta}.
In the non-resonant case, the family converges to the operator $D_0 = -\frac{d^2}{dx^2} + W$ subject to the Dirichlet boundary condition at the origin as in~\eqref{OprD0}. Both $\cH$ and $D_0$ can be interpreted as perturbations of the operator $H_0$ by point interactions at the origin.

If the potential $V$ is zero, then the family $\cH_\eps$  is uniformly bounded from below as $\eps\to 0$. This follows from the fact that the $\delta$-like perturbation is   form-bounded relative~$H_0$, with relative bound $a<1$, so that there is a $b>0$ such that
\begin{equation*}
 \forall\,\phi\in W^2_2(\Real): \qquad  \eps^{-1}\big(U(\eps^{-1}\cdot)\phi,\phi\big)\le a(H_0\phi,\phi)+b(\phi,\phi).
\end{equation*}
In this case, as we show below, the operators~$\cH_\eps$ may have at most one eigenvalue, and this eigenvalue converges to a finite limit as~$\eps\to 0$.

In contrast, when $V$ is not identically zero, although the limiting operators $\cH$ and $D_0$ are semibounded from below, the family $\cH_\eps$ is generally not uniformly bounded from below. Consequently, the family $\cH_\eps$ may exhibit eigenvalues that diverge to $-\infty$ as $\eps \to 0$; we refer to such eigenvalues as \emph{low-lying eigenvalues}.

The following result characterizes precisely when such eigenvalues may occur.

\begin{thm}
Let $\cH_\eps$ be the family of operators defined by \eqref{Heps}. Then the operators~$\cH_\eps$ admit low-lying eigenvalues as $\eps\to0$ if and only if the potential $V$ is not identically zero and
\begin{equation*}
  \int_{\Real}V\,dx\le 0.
\end{equation*}
Moreover, the number of such eigenvalues is finite.
\end{thm}

\begin{proof}
Assume that the potential~$V$ is not identically zero and that $\int_{\Real}V\,dx\le 0$. By~\cite[Th.XIII.110]{ReedSimonIV}, the operator $H=-\frac{d^2}{dx^2}+V$ has then at least one negative eigenvalue~$-\omega^2$, and we let $u$ be a corresponding normalized eigenfunction. Denote by $W_\eps(x)=W(x) + \eps^{-1} U(\eps^{-1}x)+\eps^{-2}\,V(\eps^{-1}x)$ the perturbed potential in~\eqref{Heps} and introduce the quadratic form
  \begin{equation*}
    a_\eps[\phi]=\int_\Real\bigl(|\phi'(x)|^2+W_\eps(x)|\phi(x)|^2\bigr)\,dx.
  \end{equation*}
Then the scaled function $u_\eps(x)=\eps^{-1/2}u\xe$ belongs to $L^2(\Real)$ and has norm one. A direct computation shows that
\begin{multline*}
    \eps^{2}a_\eps[u_\eps]=\int_\Real\bigl(|u'(t)|^2+V(t)|u(t)|^2\bigr)\,dt+\eps
    \int_\Real U\xe|u_\eps(x)|^2\,dx\\+\eps^{2}\int_\Real W(x)|u_\eps(x)|^2\,dx=-\omega^2(1+O(\eps)),
\end{multline*}
as $\eps\to 0$. Therefore, for all sufficiently small $\eps$, one has $a_\eps[u_\eps]\le-\frac12 \omega^2\eps^{-2}$, and we conclude from the minimax principle that there exists an eigenvalue $\lambda_\eps$ of $\cH_\eps$ such that $\lambda_\eps\le -\frac12 \omega^2 \eps^{-2}$.

Assume now that $\int_{\Real}V\,dx>0$. Then for sufficiently small $\eps>0$, the integral
\begin{equation*}
  \int_{\Real}W_\eps(x)\,dx=\eps^{-1}\int_{\Real}V(x)\,dx+\int_{\Real}U(x)\,dx+\int_{\Real}W(x)\,dx
\end{equation*}
is positive, and again by~\cite[Th.~XIII.110]{ReedSimonIV} the operator $\cH_\eps$ has no negative eigenvalues.

Let $N_\eps$ be the number of negative eigenvalues of $\cH_\eps$. It is known (see, e.g., \cite[Th.~5.3]{BerezinShubin1991}, \cite[Th.~7.5]{Simon2005}, \cite{Klaus1977, Klaus1982}) that the inequality
\begin{equation*}
  N_\eps\le 1+\int_{\Real}|x|\:|W^-_\eps(x)|\,dx
\end{equation*}
holds, where $f^-=\min\{f,0\}$ is the negative part of a function~$f$.
In view of the assumption $W \ge 0$, the negative part $W^-_\eps$ comes only from the $V$ and $U$ terms, and we estimate the integral above as follows:
\begin{align*}
  \int_{\Real}|x|\:|W^-_\eps(x)|\,dx&=\eps^{-2}\int_{\Real}|x|\:|V^-\xe+\eps U^-\xe|\,dx \\ &\le\int_{\Real}|t|\:|V^-(t)|\,dt+\eps\int_{\Real}|t|\:|U^-(t)|\,dt.
\end{align*}
The right-hand side remains bounded uniformly in small $\eps$, and thus $N_\eps$ is bounded as $\eps \to 0$, which completes the proof.
\end{proof}

Relatively (form-) bounded symmetric perturbations preserve semi-boun\-ded\-ness of the perturbed operator; see \cite[Th.~IV.4.11, Th.~VI.1.38]{Kato}. However, even if a family of self-adjoint operators~$A_\eps$ converges in the norm resolvent sense as $\eps\to0$ to a self-adjoint operator~$A$ that is bounded from below, the family~$A_\eps$ may fail to be uniformly bounded from below. Even if each operator~$A_\eps$ is individually semi-bounded, its lower bound may diverge to~$-\infty$ as~$\eps \to 0$. A classical example due to Rellich~\cite[Ex.~IV.4.14]{Kato} gives such an operator family~$A_\eps$ with a single eigenvalue tending to~$-\infty$. The family of operators~$\cH_\eps$ with $\delta'$-like perturbations provides a much stronger illustration of this effect. While $\cH_\eps$ converges in the norm resolvent topology to a self-adjoint operator that is bounded from below, the number of eigenvalues that diverge to~$-\infty$ as $\eps \to 0$ can be arbitrary but finite. In the next section, we describe the procedure for counting these low-lying eigenvalues.

%
\section{Counting the number of low-lying eigenvalues}\label{sec:counting}
%

Let us consider the Schrödinger operators
\begin{equation}\label{Talpha}
    T_\alpha = -\frac{d^2}{dx^2} + \alpha V(x), \qquad \dmn T_\alpha = W^2_2(\Real),
\end{equation}
with a real coupling constant $\alpha$. We denote by $\mathcal{R}(V)$ the set of all values of $\alpha$ for which the potential $\alpha V$ is resonant. For each non-zero function $V \in L^\infty(\Real)$ with compact support, the set $\mathcal{R}(V)$ is a countable subset of $\Real$ with accumulation points at $+\infty$ and/or $-\infty$; see~\cite{GolovatyFrontiers2019}.

We now recall the following definition~\cite{Klaus1982}. Let $A$ and $B$ be self-adjoint operators, with $B$ relatively $A$-compact. Suppose that $(a,b)$ is a spectral gap of $A$ and that $b \in \sigma_{\text{ess}}(A)$. If there exists an eigenvalue $e_\alpha$ of the perturbed operator $A + \alpha B$ in the gap $(a,b)$ for all $\alpha > 0$, and if $e_\alpha \to b-0$ as $\alpha \to 0$, then $\alpha = 0$ is called a \emph{coupling constant threshold}. Klaus~\cite{Klaus1982} established a connection between resonant potentials and such coupling constant thresholds. Both phenomena are closely related to the emergence of negative eigenvalues in Schrödinger operators.

Suppose that a Schr\"odinger operator $-\frac{d^2}{dx^2}+\cV$ has a zero-energy resonance with a corresponding half-bound state~$v$. According to~\cite[Th.~3.2]{Klaus1982}, if $\cU$ is a real-valued potential such that
\begin{equation*}
  \int_\Real \cU v^2\,dx<0,
\end{equation*}
then the perturbed operator  $H_\kappa=-\frac{d^2}{dx^2}+\cV+\kappa \cU$, $\kappa> 0$, has a coupling constant threshold at $\kappa=0$ and possesses a unique threshold eigenvalue $\lambda_\kappa$ obeying the asymptotics $\lambda_\kappa=-a^2\kappa^2+O(\kappa^3)$, as $\kappa\to 0$, where the coefficient $a$ is given by
\begin{equation}\label{AinAsymptotics}
  a=\frac1{v_-^2+v_+^2}\int_\Real \cU v^2\,dx.
\end{equation}
By reversing the direction of $\kappa$, we conclude that as $\kappa$ increases from zero, the operator~$H_\kappa$ acquires a negative eigenvalue that detaches from the bottom of the continuous spectrum.

Without loss of generality, we may assume that the support of the potential~$V$ is contained in the interval~$(-1,1)$. We now consider the spectral \textit{Regge problem} with spectral parameter $\omega$ (cf.~\cite{Shkalikov2007}):
\begin{equation}\label{ReggeProblem}
\begin{gathered}
   -\frac{d^2 u}{dx^2}+(V(x)+\omega^2)u=0, \qquad x\in(-1,1),
  \\
  \frac{d u}{dx}(-1)-\omega u(-1)=0,\quad \frac{d u}{dx}(1)+\omega u(1)=0.
\end{gathered}
\end{equation}
A complex number $\omega$ is called an eigenvalue of the Regge problem if there exists a nontrivial solution~$u$ of~\eqref{ReggeProblem}, in which case~$u$ is a corresponding eigenfunction.

%
\begin{thm}\label{thmNumberLLE}
The number of low-lying eigenvalues  of $\cH_\eps$ is equal to each of the following:
  \begin{itemize}
    \item[\textit{(i)}] the number of negative eigenvalues of the operator $T_1=-\frac{d^2}{dx^2}+V(x)$;
    \item[\textit{(ii)}] the number of points in the set $\cR(V)$ belonging to the interval $(0,1)$;
    \item[\textit{(iii)}] the number of positive eigenvalues $\omega$ of the Regge problem \eqref{ReggeProblem}.
  \end{itemize}
\end{thm}
%

\begin{proof}
\textit{(i)} Consider the family of  Schr\"{o}dinger operators
\begin{equation*}
 \cS_\eps=-\frac{d^2}{dx^2}+V(x)+ \eps U(x)+\eps^2 W(\eps x)
\end{equation*}
with domain $W^2_2(\Real)$. This family is uniformly bounded from below and converges to the operator $T_1$ in the norm resolvent sense as $\eps\to 0$. Suppose that $T_1$ has $n$ eigenvalues $-\omega_{1}^2,\dots,-\omega_{n}^2$. Then, for sufficiently small $\eps$, the operators $\cS_\eps$ have exactly $n$ eigenvalues $-\omega_{1,\eps}^2,\dots,-\omega_{n,\eps}^2$ such that $\omega_{j,\eps}\to\omega_{j}$. Since $\cH_\eps$ is unitarily equivalent to $\eps^{-2} \cS_\eps$, it follows that $\cH_\eps$ has exactly $n$ negative eigenvalues $-\eps^{-2}\omega_{1,\eps}^2,\dots,-\eps^{-2}\omega_{n,\eps}^2$, each diverging to $-\infty$ as $\eps\to0$.

\textit{(ii)$\Leftrightarrow$(i)} The operator $T_0=-\frac{d^2}{dx^2}$ has no eigenvalues.
Suppose that the set $\cR(V) \cap (0,1)$ is nonempty, and let $\alpha_1$ be its smallest element.
We write
\begin{equation*}
  T_\alpha=-\frac{d^2}{dx^2}+\alpha_1 V(x)+(\alpha-\alpha_1) V(x),
\end{equation*}
and let $v_1$ be a half-bound state corresponding to the resonant potential $\alpha_1V$. Then
\begin{equation*}
  \int_\Real V v_1^2\,dx=-\frac1{\alpha_1}\int_\Real {v'_1}^2\,dx<0,
\end{equation*}
and the operator $T_\alpha$ has an eigenvalue $\lambda_\alpha=-a_1^2(\alpha-\alpha_1)^2+O((\alpha-\alpha_1)^3)$ for $\alpha>\alpha_1$, where $a_1$ is given by \eqref{AinAsymptotics} with $v=v_1$ and $\cU=V$.

As the parameter~$\alpha$ increases, it may pass through further points in $\cR(V) \cap (0,1)$, and at each such crossing the operator~$T_\alpha$ acquires a new simple eigenvalue. Since no negative eigenvalue can get absorbed by the continuous spectrum as $\alpha$ increases~\cite{Klaus1982}, this gives a total count of the negative eigenvalues of $T_1$.

\textit{(iii)$\Leftrightarrow$(i)}
 Suppose $\omega>0$ is an eigenvalue of the Regge problem with the corresponding eigenfunction~$u$. Then $-\omega^2$ is an eigenvalue of the operator~$T_1$, with the corresponding eigenfunction
\begin{equation}\label{ReggeEigenState}
	\psi(x) = \begin{cases}
		u(-1)\,e^{\omega (x+1)}, &\text{if } x < -1, \\
		\hfil u(x),   &\text{if } |x| \le 1, \\
		u(1)\,e^{-\omega (x-1)}, &\text{if } x > 1.
	\end{cases}
\end{equation}
Conversely, if $\psi$ is an eigenfunction of $T_1$ corresponding to eigenvalue $-\omega^2$, then, since $\supp V\subset (-1,1)$, we have $\psi(x)=a_- e^{\omega x}$ for $x\le-1$ and $\psi(x)=a_+ e^{-\omega x}$ for $x\ge1$. This implies that $\psi$ satisfies the boundary conditions in \eqref{ReggeProblem}, and its restriction to $(-1,1)$ is an eigenfunction of the Regge problem \eqref{ReggeProblem} with eigenvalue $\omega>0$.
\end{proof}

Theorem~\ref{thmNumberLLE} is of practical importance because solving the Regge problem on a finite interval or computing the resonance set~$\cR(V)$ is typically much easier than directly counting the eigenvalues of a Schrödinger operator on the real line. Another useful observation is that by replacing~$V$ with~$cV$ for sufficiently large~$c > 0$, we can make the number of negative eigenvalues of~$T_1$---and hence the number of low-lying eigenvalues of~$\cH_\eps$---arbitrarily large.

The following theorem describes the two-term asymptotic expansion of the low-lying eigenvalues, which are constructed and justified in Section~\ref{sec:asymptotics}.

%
\begin{thm}\label{TheoremDeltaPrime}
	Assume  that the Schr\"odinger operator $T_1=-\frac{d^2}{dx^2}+V$ has $n$  eigenvalues $-\omega_1^2 <-\omega_2^2<\dots<-\omega_n^2<0$ with  eigenfunctions $v_1, v_2, \dots, v_n$. Then the operator family $\cH_\eps$  has $n$ low-lying eigenvalues $\lambda^\eps_1<\lambda^\eps_2< \dots < \lambda^\eps_n$  with asymptotics
	\begin{equation}\label{LowLyingAsymptotics}
		\lambda^\eps_k=-\eps^{-2}\left(\omega_k+\eps\,\dfrac{\int_\Real U|v_k|^2\,dx}{2\omega_k \|v_k\|^2}\right)^2+O(1),\quad \text{as } \eps\to +0.
	\end{equation}
	The corresponding eigenfunctions $v_{k,\eps}$ converge to zero in the weak topology.
\end{thm}
%

We mention that one of the reasons why low-lying eigenvalues do not obstruct the norm resolvent convergence of~$\cH_\eps$ is that the corresponding eigenfunctions converge weakly to zero in $L^2(\Real)$.

%
\section{Negative eigenvalues generated by $\delta$-like potentials}\label{sec:deltagenerated}
%

As shown in the previous section, the emergence of low-lying eigenvalues is caused by a $\delta'$-like perturbation, and the number of these eigenvalues is determined by the profile~$V$ of the approximating $\delta'$-like potential. However, negative eigenvalues may also arise from $\delta$-like perturbations, whether or not a $\delta'$-like component is present in the operators~$\cH_\eps$. In such cases, at most one negative eigenvalue may appear, and it always has a finite limit as $\eps\to 0$.

The Schr\"odinger operator
\begin{equation*}
    S_\alpha =  -\frac{d^2}{dx^2} + W(x) + \alpha\delta(x)
\end{equation*}
with a $\delta$-potential of intensity $\alpha\in\Real$ acts by $S_\alpha y = - y'' + Wy$ on its natural domain
\begin{align*}
	\dmn S_\alpha = \big\{y\in W^2_2(\Real\setminus{0}) \colon y(+0)=y(-0),\;\; y'(+0) - y'(-0) = \alpha y(0)\big\}.
\end{align*}
So defined $S_\alpha$ is self-adjoint and has an absolutely continuous spectrum filling the positive half-line $\Real_+$, while its negative spectrum consists of at most one eigenvalue. We recall that the unperturbed operator $S_0=-\frac{d^2}{dx^2} + W(x)$ is  non-negative.

\begin{lem}\label{LemmaOnlyDelta}
Assume $W\in L^\infty(\Real)$ is a nonnegative function of compact support. Then there exists $\alpha_0\in(-\infty,0)$ such that, for all $\alpha<\alpha_0$, the operator $S_\alpha$ has exactly one negative eigenvalue.
\end{lem}

\begin{proof}
If $W=0$, then the operator $S_\alpha$ is non-negative for $\alpha\ge0$, while for $\alpha<0$, it has a unique eigenvalue $\lambda=-\frac{\alpha^2}4$ with the normalized eigenfunction
	\begin{equation*}
		\psi_\alpha(x)=\sqrt{\tfrac{|\alpha|}2}\,e^{\frac{\alpha|x|}2},
	\end{equation*}
see \cite[Th.3.1.4]{Albeverio2edition}. For a generic $W$, we take an $\alpha<0$ and find that
	\begin{equation*}
		(S_\alpha\psi_\alpha,\psi_\alpha)
		= (-\psi''_\alpha,\psi_\alpha)+(W\psi_\alpha,\psi_\alpha)
		= -\frac{\alpha^2}{4} + \frac{|\alpha|}{2} \int_{\Real} W(x) e^{\alpha |x|}\,dx.
	\end{equation*}
	Since $\int_\Real W(x) e^{\alpha |x|}\,dx \to 0$ as $\alpha \to -\infty$ by the Lebesgue dominated convergence theorem, we conclude that the value
	\begin{equation*}
		(S_\alpha\psi_\alpha,\psi_\alpha)=\frac{|\alpha|}{4} \Bigl(\alpha + 2\int_\Real  W(x) e^{\alpha |x|}\,dx \Bigr)
	\end{equation*}
	becomes negative for negative $\alpha$ of large enough absolute value. As a result, for such~$\alpha$, the operator $S_\alpha$ has a negative eigenvalue. This eigenvalue is unique, because $S_\alpha$ is a rank-one perturbation of nonnegative operator~$S_0$.
\end{proof}	

Lemma~\ref{LemmaOnlyDelta} remains valid for positive potentials~$W$ such that
\begin{equation*}
	\int_\Real  W(x) e^{\alpha |x|}\,dx<\infty, \quad \text{for all }\alpha<0;
\end{equation*}
for example, for potentials with polynomial growth at infinity. Also, the above arguments suggest an explicit way to construct $\alpha_0$. The function
\begin{equation}\label{EqnAlphaWithStar}
	f(\alpha) = \alpha + 2\int_\Real  W(x) e^{\alpha |x|}\,dx
\end{equation}
is monotonically increasing in $\alpha\in(-\infty,0]$, $f(0)>0$ and $f$ becomes negative as $\alpha\to-\infty$.
Thus $f$ has a unique non-positive zero $\alpha_0$. Since $(S_\alpha\psi_\alpha,\psi_\alpha) = \frac{|\alpha|}{4}f(\alpha)$, we conclude that the operator $S_\alpha$ has a unique negative eigenvalue for all $\alpha<\alpha_0$.
	
\begin{exm}\rm	
Consider the family of operators
\begin{equation*}
    S_\alpha =  -\frac{d^2}{dx^2} +b^2(1+\sin x) + \alpha\delta(x)
\end{equation*}
Since $\int_\Real  (1+\sin x) e^{\alpha |x|}\,dx=-\frac{2}{\alpha}$ for $\alpha<0$, the zero of $f$ in \eqref{EqnAlphaWithStar} satisfies $\alpha^2=4b^2$. Hence $S_\alpha$ has a unique negative eigenvalue for all $\alpha<-2|b|$.
\end{exm}

\begin{exm}[{Cf.~\cite{Avakian87,FassariInglese94}}]\rm	Let $S_\alpha$ be the harmonic oscillator perturbed by the $\delta$ potential:
\begin{equation*}
  S_\alpha = -\frac{d^2}{dx^2} + kx^2 + \alpha\delta(x), \quad k>0.
\end{equation*}
In this case, the zero of $f$ is a negative root of $\alpha^4=8k$, since
\begin{equation*}
	\int_\Real  x^2 e^{\alpha |x|}\,dx=-\frac{4}{\alpha^3},\quad \alpha<0.
\end{equation*}
Therefore, the operator $S_\alpha$ has a unique negative eigenvalue  for all $\alpha<-2^{3/4}k^{1/4}$.
\end{exm}

When $V=0$, the operators
\begin{equation}\label{HepsVzero}
	\cH_\eps = -\frac{d^2}{dx^2} + W(x) + \eps^{-1} U(\eps^{-1}x),
\end{equation}
are uniformly bounded from below and converge in the norm resolvent sense to~$S_\alpha$ with
$\alpha=\int_{\Real}U\,dx$. This convergence, in particular, implies the convergence of negative eigenvalues; our next objective is to obtain a more precise asymptotic formula (proved in Section~\ref{sec:asymptotics}).

%
\begin{thm}\label{TheoremDelta}
    Suppose that $W$ and $U$ are $L^\infty(\Real)$-functions of compact support, and that $W$ is nonnegative. If $\int_{\Real}U\,dx<\alpha_0$, where the threshold value~$\alpha_0$ is the root of~\eqref{EqnAlphaWithStar}, then the operator $\cH_\eps$ of \eqref{HepsVzero}
	has a unique negative eigenvalue $\lambda_\eps$ satisfying the asymptotics
	\begin{multline}\label{DeltaBoundStateAsymp}
		\lambda_\eps=\lambda+\eps \bigg(\frac12\psi(0)^2\iint_{\Real^2} U(t)|t-\tau|U(\tau)\,d\tau\,dt
		\\
		+ \psi(0)\big(\psi'(-0)+\psi'(+0)\big)\int_{\Real} t U(t)\,dt\bigg)+O(\eps^2),\quad \eps\to 0.
	\end{multline}
	Here, $\psi$ is a real-valued, $L^2(\Real)$-normalized eigenfunction of $S_\alpha$, with $\alpha=\int_{\Real}U\,dx$, corresponds to the unique negative eigenvalue $\lambda$.

	Moreover, the normalized eigenfunctions $\psi_\eps$ of $S_\eps$ can be chosen in such a way that
	$\psi_\eps\to \psi$ in $L^2(\Real)$.
\end{thm}
%

If the potential $W$ is even, then the ground state $\lambda_\eps$ has asymptotics
\begin{equation*}
	\lambda_\eps=\lambda+\frac12\,\eps\psi(0)^2\iint_{\Real^2} U(t)|t-\tau|U(\tau)\,d\tau\,dt+O(\eps^2).
\end{equation*}
The eigenfunction $\psi$ is also even and therefore $\psi'(-0)+\psi'(+0)=0$.
If $W=0$ and $\int_{\Real}U\,dx<0$, the asymptotic formula \eqref{DeltaBoundStateAsymp} becomes
\begin{equation*}
	\lambda_\eps=-\frac14\left(\int_{\Real} U(t)\,dt+\frac12\:\eps\iint_{\Real^2} U(t)|t-\tau|U(\tau)\,d\tau\,dt\right)^2+O(\eps^2)
\end{equation*}
and coinsides with the Abarbanel--Callan--Goldberger formula up to the factor $\eps^2$. The formula arises when studying the weakly coupled Hamiltonians $-\frac{d^2}{dx^2} +\gamma U$, their negative eigenvalues and the absorption of such eigenvalues, as $\gamma\to 0$, by a continuous spectrum~\cite{Simon1976}.

Now suppose that the potential $V$ is nonzero. If $V$ is non-resonant, then the behavior of the negative spectrum of $\cH_\eps$ is described by Theorem~\ref{TheoremDeltaPrime}. However, if the shape $V$ of the $\delta'$-perturbation is resonant, then under certain conditions on the $\delta$-perturbation, the operator $\cH_\eps$ may have---in addition to low-lying eigenvalues---an extra eigenvalue that has a finite limit as $\eps\to 0$. We recall that in the resonant case, the norm resovent limit of $\cH_\eps$ as $\eps\to0$ is the operator~$\cH$ given by~\eqref{Hthetaeta}, with constants $\theta$ and $\eta$ determined by $V$ and $U$ via~\eqref{ThetaEta}.

\begin{lem}
Let $W$ be a non-negative function in $L^\infty(\Real)$ of compact support.
If $\eta\theta<0$ and the condition
\begin{equation}\label{CondWThetaW}
   \int_0^{+\infty}\big(W(-x)+\theta^2 W(x)\big)\,dx<\frac{|\eta\theta|}{2}
\end{equation}
holds, then the operator $\cH$ defined by \eqref{Hthetaeta} has a unique negative eigenvalue.
\end{lem}

\begin{proof}
Assume first that $W=0$. Integration by parts, on account of the interface conditions, yields
\begin{equation*}
    ( \cH y, y) = \|y'\|^2 + \theta\eta |y(-0)|^2,
\end{equation*}
thus for $\theta\eta\ge0$ the operator $\cH$ is non-negative. In contrast, if $\theta\eta<0$, then~$\cH$  has a unique eigenvalue
 \begin{equation}\label{EigenvalueH}
 \lambda =-\frac{\eta^2\theta^2}{(\theta^2+1)^2}
\end{equation}
with the normalized eigenfunction
\begin{equation*}
  \Psi(x)= \tfrac{\sqrt{2|\eta\theta|}}{\theta^2+1}\cdot
  \begin{cases}
   \phantom{A} e^{-\frac{\eta\theta}{\theta^2+1} x}& \text{for } x<0,\\
    \theta e^{\frac{\eta\theta}{\theta^2+1} x}& \text{for } x>0,
  \end{cases}
\end{equation*}
as can be verified by straightforward calculations.

Now consider the case of and arbitrary $W\in L^\infty(\Real)$, and let the constants $\theta$ and $\eta$ from~\eqref{ThetaEta} satisfy $\theta\eta<0$. Using the function~$\Psi$ defined above, we find that
  \begin{equation*}
    (\cH\Psi,\Psi)=(W\Psi,\Psi)-
    \frac{\eta^2\theta^2}{(\theta^2+1)^2}.
  \end{equation*}
Therefore, $\cH$ has a negative eigenvalue if
\begin{equation*}
 (W\Psi,\Psi)<\frac{\eta^2\theta^2}{(\theta^2+1)^2}.
\end{equation*}
This inequality is equivalent to
\begin{equation*}
  \int_0^{+\infty}\big(W(-x)+\theta^2 W(x)\big)\,e^{-\frac{2|\eta\theta|}{\theta^2+1} x}\,dx<\frac{|\eta\theta|}{2},
\end{equation*}
which is guaranteed under condition \eqref{CondWThetaW}.
\end{proof}	

\begin{thm}\label{TheoremWDelta}
Assume that $V$ is resonant with a half-bound state $v$, and that the potentials $W$ and $U$ satisfy the conditions $W\ge0$ and
\begin{equation}\label{CondEVexistsPrime}
  \int_0^{+\infty}\big(v_-^2W(-x)+ v_+^2 W(x)\big)\,dx<-\frac12 \int_\Real U v^2\,dx.
\end{equation}
Then, for $\eps$ small enough, the operator $\cH_\eps$ has a negative eigenvalue $\lambda_\eps$ converging, as $\eps\to 0$, to the negative eigenvalue of the operator $\cH$ defined by~\eqref{Hthetaeta}, where the parameters $\theta$ and $\eta$ given by \eqref{ThetaEta}.

If $W=0$ and $\int_\Real U v^2\,dx<0$, then this eigenvalue $\lambda_\eps$ has asymptotics
\begin{equation}\label{NFEW0}
  \lambda_\eps=-\frac1{(v_-^2+ v_+^2)^2}\left(\int_\Real U v^2\,dx\right)^2+O(\eps), \quad \eps\to 0.
\end{equation}
\end{thm}

\begin{proof}
Inequality \eqref{CondEVexistsPrime} and asymptotic formula \eqref{NFEW0} are equivalent forms of \eqref{CondWThetaW} and \eqref{EigenvalueH} when evaluated for the specific values $\theta$ and $\eta$. In addition, inequality~\eqref{CondEVexistsPrime} ensures that $\eta\theta<0$. Indeed, it implies that $\int_\Real U v^2\,dx <0$, and since
\begin{equation*}
  \eta\theta=\frac{1}{v_-^2}\int_\Real U v^2\,dx,
\end{equation*}
we conclude that $\eta\theta < 0$. The convergence $\cH_\eps \to \cH$ in the norm resolvent sense as $\eps \to 0$ then guarantees that $\cH_\eps$ has a negative eigenvalue $\lambda_\eps$ approaching the unique negative eigenvalue of~$\cH$.
\end{proof}

\section{Asymptotic expansions of eigenvalues}
\label{sec:asymptotics}
In this section, we derive asymptotic formulas \eqref{LowLyingAsymptotics} and \eqref{DeltaBoundStateAsymp} by constructing and justifying formal asymptotic expansions of the eigenvalues. For the sake of definiteness, we assume that the supports of $U$ and $V$ are contained in~$(-1,1)$.

\subsection{Formal asymptotics}
We start with asymptotics \eqref{DeltaBoundStateAsymp}. The equation
\begin{equation*}
	-\frac{d^2y_\eps}{dx^2}+\left(W(x)+\eps^{-1}\,U(\eps^{-1}x)\right)y_\eps=\lambda_\eps y_\eps
\end{equation*}
on $\Real\setminus (-\eps,\eps)$ reads
\begin{equation}\label{SpectralEqnOut}
	-\frac{d^2y_\eps}{dx^2}+W(x)y_\eps=\lambda_\eps y_\eps,
\end{equation}
while after rescaling $(-\eps,\eps)$ to $(-1,1)$ and introducing  $w_\eps(t)=y_\eps(\eps t)$, one gets
\begin{equation}\label{SpectralEqnIn}
	-\frac{d^2w_\eps}{dt^2} + \eps^2 W(\eps t)w_\eps + \eps U(t) w_\eps=\eps^2\lambda_\eps w_\eps
\end{equation}
on $(-1,1)$. Also, the components $y_\eps$ and $w_\eps$ must satisfy the matching conditions
\begin{equation*}
	w_\eps(\pm 1)=y_\eps(\pm\eps), \qquad w_\eps'(\pm 1)=\eps y_\eps'(\pm\eps).
\end{equation*}
We look for approximations of eigenvalues and eigenfunctions of the form
\begin{align}\label{OmegaExp}
	\lambda_\eps&\sim \lambda_0+\eps \lambda_1,\\\label{YepsExp}
	y_\eps(x)&\sim Y_\eps(x)=
	\begin{cases}
		y_0(x)+\eps y_1(x),& \text{if } |x|>\eps,\\
		w_0(\eps^{-1}x)+\eps w_1(\eps^{-1}x)+\eps^2w_2(\eps^{-1}x),& \text{if } |x|<\eps,
	\end{cases}
\end{align}
where $y_0\neq 0$. Substituting the approximations into \eqref{SpectralEqnOut} and \eqref{SpectralEqnIn}, we find that $y_0$ and $y_1$ satisfy the equations
\begin{equation*}
	-\frac{d^2y_0}{dx^2}+W(x)y_0=\lambda_0 y_0,\qquad
	-\frac{d^2y_1}{dx^2}+W(x)y_1=\lambda_0 y_1+\lambda_1 y_0
\end{equation*}
on $\Real\setminus\{0\}$, while the fast-variable components $w_0$, $w_1$, and $w_2$ are solutions to the boundary value problems
\begin{align}
\label{W0problem}
  &\frac{d^2 w_0}{dt^2}=0,\qquad \frac{d w_0}{dt}(-1)=0, \quad \frac{d w_0}{dt}(1)=0;
\\\label{W1problem}
  &\frac{d^2 w_1}{dt^2}=U(t)w_0, \qquad \frac{d w_1}{dt}(-1)=y_0'(-0), \quad\frac{d w_1}{dt}(1)=y_0'(+0);
\\\label{W2problem}
  &
\begin{aligned}
\frac{d^2 w_2}{dt^2}&=U(t)w_1+(W(0)-\lambda_0)w_0,\\ &\frac{d w_2}{dt}(-1)=y_1'(-0)-y_0''(-0), \quad\frac{d w_2}{dt}(1)=y_1'(+0)+y_0''(+0).
\end{aligned}
\end{align}
Furthermore, the equalities
\begin{gather}\label{W0conds}
  w_0(-1)=y_0(-0), \quad w_0(1)=y_0(+0),
\\\label{W1conds}
w_1(-1)=y_1(-0)-y_0'(-0), \quad w_1(1)=y_1(+0)+y_0'(+0)
\end{gather}
hold. In view of \eqref{W0problem} and \eqref{W0conds}, $w_0$ is a constant function and therefore $y_0(+0)=y_0(-0)$.
Set $w_0(t)=y_0(0)$.  Problem \eqref{W1problem} can be solved if and only if
\begin{equation*}
  \frac{d w_1}{dt}(1)-\frac{d w_1}{dt}(-1)=y_0(0)\int_{-1}^1U(\tau)\,d\tau,
\end{equation*}
which yields the second interface condition for $y_0$:
\begin{equation}\label{Y0primeJump}
    y_0'(+0)-y_0'(-0)=\alpha y_0(0),\qquad \alpha=\int_\Real U(\tau)\,d\tau.
\end{equation}
Therefore, the leading terms $y_0$ and $\lambda_0$ of \eqref{OmegaExp}, \eqref{YepsExp} solve the problem
\begin{equation*}
\begin{gathered}
  -\frac{d^2y_0}{dx^2}+W(x)y_0=\lambda_0 y_0\quad\text{on }\Real\setminus\{0\},\\
  y_0(+0)=y_0(-0), \quad y_0'(+0)-y_0'(-0)=\alpha y_0(0).
\end{gathered}
\end{equation*}
Since $y_0$ must be a non-trivial solution, we conclude that $\lambda_0$ is an eigenvalue of $S_\alpha$ and $y_0$ is the corresponding (real-valued) eigenfunction; we denote it by $\psi$ and normalize by $\|\psi\|=1$.
Moreover, problem \eqref{W1problem} is solvable now and the solution $w_1$ is defined up to a constant.

Integrating twice the equation for $w_1$ and using the relations $w_1'(-1) = \psi'(-0)$ and $w_1(-1) = y_1(-0) - \psi'(-0)$, we arrive at the formula
\begin{equation}\label{W1}
	w_1(t)=\psi(0)\int_{-1}^t(t-\tau)U(\tau)\,d\tau + \psi'(-0)t + y_1(-1),
\end{equation}
which on account of $w_1(1)=y_1(+0) + \psi'(+0) $ yields
\begin{equation*}
y_1(+0) + \psi'(+0) = \psi(0)\int_{-1}^1(1-\tau)U(\tau)\,d\tau + \psi'(-0) + y_1(-0).
\end{equation*}
Set $\alpha_1=\int_\Real\tau U(\tau)\,d\tau$. Then $ y_1(+0)-y_1(-0)=-\alpha_1\psi(0)$, by \eqref{Y0primeJump}.

To get the second interface relation for $y_1$, we integrate the equation for $w_2$ and find that
\begin{equation*}
  w_2'(1)-w_2'(-1)=\int_{-1}^1U(t)w_1(t)\,dt+ 2(W(0)-\lambda_0)\psi(0).
\end{equation*}
We assume that $W$ is continuous in a vicinity of $x=0$ and this implies that
\begin{equation*}
 \psi''(+0)=\psi''(-0)=(W(0)-\lambda_0)\psi(0).
\end{equation*}
Combining this with the boundary conditions in \eqref{W2problem} and \eqref{W1}, we obtain
\begin{equation*}
    y_1'(+0)-y_1'(-0)=\alpha y_1(-0)
    +\alpha_1\psi'(-0)+ \gamma \psi(0),
\end{equation*}
where
\begin{equation*}
  \gamma=\int_\Real\int_{-\infty}^tU(t)(t-\tau)U(\tau)\,d\tau\,dt.
\end{equation*}
So, we get the boundary value problem for $y_1$:
\begin{gather}\label{EqnY1}
  -\frac{d^2y_1}{dx^2}+W(x)y_1=\lambda_0 y_1+\lambda_1 \psi\quad\text{on }\Real\setminus\{0\},
  \\\label{FirstCondY1}
  y_1(+0)-y_1(-0)=-\alpha_1\psi(0),
  \\\label{SecondCondY1}
  y_1'(+0)-y_1'(-0)-\alpha y_1(-0)=\alpha_1\psi'(-0)+ \gamma \psi(0).
\end{gather}
Observe that the solution $y_1$, if exists, is determined up to adding a multiple of~$\psi$; therefore, by the Fredholm alternative, the above non-homogeneous problem is solvable only when some extra conditions are met. To derive them, we multiply equation~\eqref{EqnY1} by the eigenfunction~$\psi$ and then integrate by parts twice to get
\begin{equation*}
	\psi(0)\big(y_1'(+0)-y_1'(-0)-\alpha y_1(-0)\big)-\psi'(+0)\big(y_1(+0)-y_1(-0)\big)=\lambda_1\int_{\Real} \psi^2(x)\,dx.
\end{equation*}
Relations \eqref{FirstCondY1} and \eqref{SecondCondY1} result in the expression
\begin{equation*}
	\lambda_1=\gamma  \psi^2(0)+\alpha_1 \psi(0)(\psi'(-0)+\psi'(+0))
\end{equation*}
for the second term in asymptotics \eqref{DeltaBoundStateAsymp}.

We seek an approximation of the low-lying eigenvalues and the corresponding eigenfunctions of the form
\begin{align}\label{ApprOmega}
\sqrt{-\lambda_\eps}&\sim \omega_\eps=\eps^{-1}(\omega+\eps\kappa),
\\\label{ApprY}
y_\eps(x)&\sim Y_\eps(x)=
  \begin{cases}
    e^{\frac{(\omega+\eps\kappa)x}{\eps}}, &\text{if } x<-\eps,
    \\
    u(\eps^{-1}x)+\eps w (\eps^{-1}x),& \text{if } |x|<\eps,
    \\
    (a+b\eps)e^{-\frac{(\omega+\eps\kappa)x}{\eps}}, &\text{if } x>\eps,
  \end{cases}
\end{align}
where $\omega>0$. On the region $\Real\setminus (-\eps,\eps)$, the function $Y_\eps$ satisfies equation \eqref{SpectralEqnOut} up to terms with a small norm in $L^2(\Real)$. For instance,  only the term $W(x) e^{\frac{(\omega+\eps\kappa)x}{\eps}}$ remains if $x<-\eps$, and its norm in $L^2(-\infty,0)$ is of order $O(\eps^{1/2})$. By substituting $Y_\eps$ into  \eqref{SpectralEqnIn} and matching the terms at $x=\pm \eps$, we obtain
\begin{equation}\label{YintoEqnConds}
\begin{gathered}
  -\frac{d^2 u}{dt^2}+(V(t)+\omega^2)u=0, \quad -\frac{d^2 w }{dt^2}+(V(t)+\omega^2)w =-(2\omega\kappa+U(t))u,
  \\
  u(-1)=e^{-\omega}, \quad \frac{d u}{dt}(-1)=\omega e^{-\omega}, \quad  u(1)=ae^{-\omega}, \quad \frac{d u}{dt}(1)=-a\omega e^{-\omega},
  \\
  \begin{aligned}
     w (-1)=&-\kappa e^{-\omega}, \quad\frac{d w }{dt}(-1)=\kappa(1-\omega) e^{-\omega}, \\
  &w (1)=(b -a\kappa)e^{-\omega}, \quad \frac{d w }{dt}(1)=(a\kappa(\omega-1)-b\omega) e^{-\omega}.
  \end{aligned}
\end{gathered}
\end{equation}
The relations for $u$ can be written as
\begin{gather}\label{ReggeEqn}
   -\frac{d^2 u}{dt^2}+(V(t)+\omega^2)u=0, \qquad t\in(-1,1),
  \\ \label{ReggeConds}
  \frac{d u}{dt}(-1)-\omega u(-1)=0,\quad \frac{d u}{dt}(1)+\omega u(1)=0,
\end{gather}
which is the Regge problem \eqref{ReggeProblem}. We assume that $\omega$ is an eigenvalue of the problem  with real-valued eigenfunction~$u$. Recall that  $-\omega^2$ is an eigenvalue of the operator~$T_1$ of~\eqref{Talpha} with the eigenfunction $\psi$ given by \eqref{ReggeEigenState}.
We also have $a=u(1)/u(-1)$.

From \eqref{YintoEqnConds} we similarly obtain the problem for $w$:
\begin{gather}\label{ReggeEqnV1}
  -\frac{d^2 w }{dt^2}+(V(t)+\omega^2)w =-(2\omega\kappa+U(t))u, \qquad t\in(-1,1),
  \\\label{ReggeCondsV1}
  \frac{dw }{dt}(-1)-\omega w (-1)=\kappa u(-1),\quad \frac{dw }{dt}(1)+\omega w (1)=-\kappa u(1).
\end{gather}
The problem is  generally unsolvable because $\omega$ is an eigenvalue of the homogeneous problem \eqref{ReggeEqn}, \eqref{ReggeConds}. In this situation, however, the free parameter $\kappa$ can be chosen so that the problem admits solutions.

Solvability condition of  \eqref{ReggeEqnV1}, \eqref{ReggeCondsV1} has the form
\begin{equation*}
  \kappa=-\frac{\int_{-1}^1Uu^2\,dt}{u^2(-1)+u^2(1)+2\omega\|u\|^2}.
\end{equation*}
When $\omega>0$, then the denominator can be written as $2\omega \|\psi\|^2$, with the eigenstate $\psi$ of~\eqref{ReggeEigenState} resulting in
\[
	\kappa = -\frac{\int_{\Real} U |\psi^2(x)|\,dx}{2 \omega \|\psi\|^2}.
\]
With $\kappa$ as above, there exists a solution $w$ of \eqref{ReggeEqnV1}, \eqref{ReggeCondsV1} defined up to the additive term $cu$. Finally, we can calculate
\begin{equation*}
 b=\frac{u(-1)w (1)-u(1)w (-1)}{u^2(-1)}.
\end{equation*}
Observe that the right hand side of the latter expression is independent of the chosen partial solution $w$. Hence, we have formally obtained asymptotics \eqref{LowLyingAsymptotics}.


\subsection{Justification of asymptotics}
We now justify the asymptotic representations for $\lambda_\eps$ and $y_\eps$ by constructing a so-called \emph{quasimode} for the operator $\cH_\eps$.

Let $A$ be a self-adjoint operator in a Hilbert space $L$. A pair $(\mu, \phi)\in \Real\times \dmn A$ is called a \emph{quasimode} of $A$ with accuracy $\epsilon$ if $\|\phi\|_L = 1$ and $\|(A-\mu I)\phi\|_L \leq \epsilon$.

\begin{lem}[\hglue-0.1pt{\cite[p.139]{PDEVinitiSpringer}}]
	Assume $(\mu, \phi)$ is a quasimode of $A$ with accuracy $\epsilon>0$  and that the spectrum of $A$ in  the interval
	$[\mu-\epsilon, \mu+\epsilon]$ is discrete. Then there exists an eigenvalue $\lambda$ of $A$ such that $|\lambda-\mu|\leq\epsilon$.
\end{lem}

Moreover, if the interval $[\mu - \Delta, \mu + \Delta]$ contains precisely one simple eigenvalue $\lambda$ with normalized eigenvector $u$, then
\begin{equation}\label{QuasimodeEst}
  \|\phi - e^{ia} u\| \leq 2\epsilon\Delta^{-1}
\end{equation}
for some real number $a$.

A quasimode of $\cH_\eps$ can be constructed based on the approximation $Y_\eps$. Note, however, that the function $Y_\eps$ defined by~\eqref{YepsExp} is not smooth enough to belong to the domain $\dmn \cH_\eps$, as it has jump discontinuities at the points $x = \pm \eps$. Nevertheless, all these jumps are small due to the construction; namely,
\begin{equation}\label{JumpsEst}
  \big|[Y_\eps]_{-\eps}\big| + \big|[Y_\eps]_{\eps}\big| + \big|[Y_\eps']_{-\eps}\big| + \big|[Y_\eps']_{\eps}\big| \le c \eps^2,
\end{equation}
where $[\,\cdot\,]_x$ denotes the jump of a function at the point $x$.

	Suppose the functions $\zeta$ and $\eta$ are smooth outside the origin, have compact supports contained in $[0,\infty)$, and $\zeta(+0)=1$, $\zeta'(+0)=0$, $\eta(+0)=0$, $\eta'(+0)=1$. We introduce the function
	\begin{equation*}
		r_\eps(x)=[Y_\eps]_{-\eps}\, \zeta(-x-\eps)-[Y_\eps']_{-\eps}\,\eta(-x-\eps)-[Y_\eps]_{\eps}\,\zeta(x-\eps)-[Y_\eps']_{\eps}\,\eta(x-\eps),
	\end{equation*}
	which has the jumps at $\pm\eps$ that are negative of those of $Y_\eps$. Therefore, the function $\hat{y}_\eps=Y_\eps+r_\eps$
	is continuous on~$\Real$ along with its derivative and consequently belongs to $W_2^2(\Real)$. Moreover, $\|\hat{y}_\eps\|=1+O(\eps)$ as $\eps\to 0$, because the main term $y_0=\psi$ is normalized in $L^2(\Real)$. The corrector function $r_\eps$ is small, because $r_\eps$ is identically zero on $(-\eps,\eps)$, and \eqref{JumpsEst} makes it obvious that
	\begin{equation*}
		\max_{|x|\ge \eps}\big|r_{\eps}^{(k)}(x)\big|\le c\eps^2,\qquad k=0,1,2.
	\end{equation*}
	A straightforward computation shows that a pair $\big(\lambda_0+\lambda_1\eps,\hat{y}_\eps\big)$ is a quasimode of~$\cH_\eps$ with accuracy of order $O(\eps^2)$, i.e., $\|\cH_\eps\hat{y}_\eps-(\lambda_0+\lambda_1\eps)\hat{y}_\eps \|\le c\eps^2\|\hat{y}_\eps\|$. Hence,
	\begin{equation*}
		|\lambda_\eps-(\lambda_0+\lambda_1\eps)|\le C\eps^2,
	\end{equation*}
	where $\lambda_\eps$ is an eigenvalue of $\cH_\eps$. Since $\lambda_\eps$ is a simple eigenvalue, the corresponding eigenfunction $y_\eps$ can be chosen so that $y_\eps\to y_0$ in $L^2(\Real)$, by \eqref{QuasimodeEst}.
Theorem~\ref{TheoremDelta} is proved.

The proof of Theorem~\ref{TheoremDeltaPrime} is similar, with one key difference. The approximation given by~\eqref{ApprOmega} and~\eqref{ApprY} is not sufficient to construct a quasimode of $\cH_\eps$ with sufficiently small accuracy. It is therefore necessary to refine the approximation as follows:
\begin{align*}
\sqrt{-\lambda_\eps} &\sim \omega_\eps = \eps^{-1}(\omega + \eps\kappa_1 + \eps^2\kappa_2 + \eps^3\kappa_3), \\
y_\eps(x) &\sim Y_\eps(x) =
  \begin{cases}
    e^{\omega_\eps x}, & \text{if } x < -\eps, \\
    u\xe + \eps u_1 \xe + \eps^2 u_2 \xe + \eps^3 u_3 \xe, & \text{if } |x| < \eps, \\
    (a + a_1\eps + a_2\eps^2 + a_3\eps^3)e^{-\omega_\eps x}, & \text{if } x > \eps,
  \end{cases}
\end{align*}
after which the construction proceeds as in the case of Theorem~\ref{TheoremDelta}.


\section{Concluding remarks}


Exactly solvable models of quantum mechanics with point interactions provide useful approximations for short-range interactions between particles. In particular, $\delta$- and $\delta'$-type potentials serve as mathematical abstractions representing idealized phenomena such as sharply localized charges or dipoles. However, while these models facilitate rigorous quantitative analysis, they do not always preserve the qualitative behavior of the corresponding regular systems.

For example, Hamiltonians with $\delta$-potentials capture well both the spectral and qualitative behavior of regular Hamiltonians with sharply localized attractive wells; in particular, the unique bound state of the regular system persists in the limit model. In contrast, Hamiltonians with localized dipoles may possess arbitrarily many negative eigenvalues, while the corresponding exactly solvable model with $\delta'_\theta$-potential has none. We stress that, in both settings, the exactly solvable models arise as norm resolvent limits of families of regular Schrödinger operators as the regularization parameter tends to zero.

Our results show that, despite norm resolvent convergence to a bounded below operator with $\delta'_\theta$-potential, the family of approximating Hamiltonians with localized dipoles is not uniformly bounded below: their low-lying eigenvalues diverge to $-\infty$ as the regularization parameter tends to zero. This demonstrates that even the strongest form of convergence of Hamiltonians does not ensure that the spectral or qualitative properties of the real physical models are reflected in the idealized limit. Therefore, while exactly solvable models offer powerful tools for quantitative analysis, caution is needed when interpreting their qualitative features as representative of the physical
systems they are meant to approximate.




\begin{thebibliography}{10}

\bibitem{Albeverio2edition}
    S. Albeverio,  F. Gesztesy,  R. H{\o}egh-Krohn, and  H. Holden,
    {\it Solvable Models in Quantum Mechanics}. With an Appendix by Pavel Exner. 2nd revised edn,
    Providence, RI: AMS Chelsea Publishing, 2005. P.~488.

\bibitem{AlbeverioKurasov}
    S. Albeverio and P. Kurasov,
    {\it Singular Perturbations of Differential Operators.
         Solvable Schr\"{o}dinger Type Operators},
    (London Mathematical Society Lecture Note Series vol 271)
    Cambridge: Cambridge University Press, 1999. P.~429.

\bibitem{SebRMP}
    P. \v{S}eba, {Some remarks on the $\delta'$-interaction in one dimension},
    \textit{Rep. Math. Phys.} \textbf{24} (1986), no.~1, 111--120.

\bibitem{ChristianZolotarIermak03}
    P.~L.~Christiansen,  H.~C.~Arnbak,  A.~V.~Zolotaryuk,  V.~N.~Ermakov, Y.~B.~Gaididei.
    On the existence of resonances in the transmission probability for interactions  arising from derivatives of Dirac’s delta function.
    \textit{J. Phys. A: Math. Gen.} {\bf 36} (2003), 7589--7600.

\bibitem{Zolotaryuk08}
    A. V. Zolotaryuk,
    {Two-parametric resonant tunneling across the $\delta'(x)$ potential}.
    \textit{Adv. Sci. Lett.} {\bf 1} (2008), 187--191.

\bibitem{Zolotaryuk09}
    A. V. Zolotaryuk,
    Point interactions of the dipole type defined through a three-parametric power regularization.
    \textit{J. Phys. A: Math. Theor.} {\bf 43} (2010), 105302.

\bibitem{Zolotaryuk10}
    A. V. Zolotaryuk,
    Boundary conditions for the states with resonant tunnelling across the $\delta'$-potential.
    \textit{Phys. Lett. A}  {\bf 374} (2010), no.~15--16, 1636--1641.


\bibitem{GolovatyHrynivJPA2010}
    Yu. D. Golovaty and R. O. Hryniv,
    {On norm resolvent convergence of Schr\"{o}dinger operators with $\delta'$-like potentials},
   \textit{J. Phys.~A: Math. and Theor.} \textbf{43} (2010) 155204 (14pp) (A Corrigendum: \textit{J. Phys. A: Math. Theor.} \textbf{44} (2011) 049802)


\bibitem{Klaus1982}
    M.~Klaus,
    Some applications of the Birman--Schwinger principle,
    \textit{Helv. Phys. Acta} \textbf{55} (1982/83), no.~1, 49--68.

\bibitem{HintonKlausShaw91a}
    D.~B. Hinton, M.~Klaus, and J.~K. Shaw,
    Embedded half-bound states for potentials of Wigner–von Neumann type,
    \textit{Proc. Lond. Math. Soc. (3)} \textbf{62} (1991), no.~3, 607--646.

\bibitem{HintonKlausShaw91b}
    D.~B. Hinton, M.~Klaus, and J.~K. Shaw,
    Half-bound states and Levinson's theorem for discrete systems,
    \textit{SIAM J. Math. Anal.} \textbf{22} (1991), no.~3, 754--768.

\bibitem{GolovatyManko2009} Yu.~Golovaty,  S.~Man'ko,
    Solvable models for the Schr\"odinger operators
    with $\delta'$-like potentials     { Ukr. Math. Bulletin} {6} (2) (2009), 169--203. Available online:
    \url{https://doi.org/10.48550/arXiv.0909.1034}


\bibitem{GolovatyHrynivProcEdinburgh2013}
    Yu. D. Golovaty and R. O. Hryniv,
    Norm resolvent convergence of singularly scaled Schr\"{o}dinger operators and $\delta'$-potentials.
    \textit{Proc. Royal Soc. Edinburgh: Sect. A Math.} \textbf{143} (2013),  791-816.

\bibitem{KurElaMSI1993}
    P.~Kurasov and N. Elander,
    On the $\delta'$-interactions in one dimension,
    \textit{Technical report}, MSI, Stockholm, 1993.


\bibitem{Kurasov1996}
    P.~Kurasov,
    Distribution theory for discontinuous test functions and differential operators with generalized coefficients,
    \textit{J. Math. Anal. Appl.} \textbf{201} (1996), 297--323.

\bibitem{GolovatyMFAT2012}
    Yu. Golovaty,
    Schr\"odinger operators with $(\alpha\delta'+\beta\delta)$-like potentials: norm resolvent convergence and solvable models,
    \textit{Methods Funct. Anal. Topology}, \textbf{18} (2012), no.~3, pp. 243--255. Available online:
    \url{http://mfat.imath.kiev.ua/article/?id=633}


\bibitem{GolovatyIEOT2012}
    Yu. Golovaty,
    1D Schr\"odinger operators with short range interactions: two-Scale regularization of distributional potentials,
    \textit{Integr. Equ. Oper. Theory}, \textbf{75} 2013, no.~3, pp.~341--362.

\bibitem{GadellaHerasNietoas2009}
    M.~Gadella, F. J. H. Heras, J. Negro, and L. M. Nieto,
    A delta well with a mass jump,
    \textit{Journal of Physics A: Mathematical and Theoretical}, \textbf{42} (2009), id.~46, 465207.

\bibitem{GadellaNegroNieto2009}
    M.~Gadella, J. Negro, and L. M. Nieto,
    Bound states and scattering coefficients of the $-\alpha\delta'(x)+ \beta\delta(x)$ potential,
    \textit{Physics Letters A} \textbf{373} (2009), no.~15, 1310-1313.

\bibitem{GadellaGlasserNieto2011}
    M. Gadella, M. L. Glasser, and L. M. Nieto,
    One dimensional models with a singular potential of the type $-\alpha\delta'(x)+ \beta\delta(x)$,
    \textit{Int. J. Theor. Phys.} \textbf{50} (2011), no.~7, 2144--2152.

\bibitem{NizhFAA2003}
     L.P. Nizhnik,
     A Schr\"odinger operator with $\delta'$-interaction,
     \textit{Functional Analysis and Its Applications}, \textbf{37} (2003), no.~1, 72--74

\bibitem{NizhFAA2006}
    L. P. Nizhnik,
    A one-dimensional Schr\"odinger operator with point interactions on Sobolev spaces,
    \textit{Functional Analysis and Its Applications}, \textbf{40} (2006), 143--147.

\bibitem{LunardiManzoni2014}
    M. Cal\c{c}ada,  J. T. Lunardi, L. A. Manzoni, and W. Monteiro,
    Distributional approach to point interactions in one-dimensional quantum mechanics,
    \textit{Frontiers in Physics}, \textbf{2} (2014), id.~23.

\bibitem{LunardiManzoni2016}
    M. A. Lee,  J. T. Lunardi,  L. A. Manzoni,  and  E. A. Nyquist,
    Double general point interactions: symmetry and tunneling times,
    \textit{Frontiers in Physics}, \textbf{4} (2016), id.~10.

\bibitem{ExnerNeidhardtZagrebnov2001}
    P.~Exner, H.~Neidhardt, V.~A.~Zagrebnov.
    Potential approximations to $\delta'$: an inverse Klauder phenomenon with norm-resolvent convergence.
    textit{Comm. Math. Phys.}  \textbf{224} (2001), no.~3, 593--612.

\bibitem{CheonShigehara1998}
    T. Cheon and T. Shigehara,
    Realizing discontinuous wave functions with renormalized short-range potentials.
    \textit{Phys. Lett. A}, \textbf{243} (1998), no.~3, 111--116.

\bibitem{ZolotaryukThreeDelta17}
    A. V. Zolotaryuk,
    {Families of one-point interactions resulting from the squeezing limit of the sum of two- and
    three-delta-like potentials},
     \textit{J. Phys. A: Math. Theor.}, \textbf{50} (2017), no.~22, id.~225303.

\bibitem{FassariRinaldi2009}
    S. Fassari and F. Rinaldi,
     \textit{On the spectrum of the Schrödinger Hamiltonian with a particular configuration of three one-dimensional point interactions.}
     \textit{Rep. Math. Phys.} \textbf{64} (2009), no.~3, 367--393.

\bibitem{AlbeverioFassariRinaldi2013}
    S. Albeverio, S. Fassari, and F. Rinaldi,
    A remarkable spectral feature of the Schr\"odinger Hamiltonian of the harmonic oscillator perturbed by an attractive $\delta'$-interaction centred at the origin: double degeneracy and level crossing,
    \textit{J. Phys. A: Math. Theor.}, \textbf{46} (2013), no.~38, id.~385305.

\bibitem{AlbeverioFassariRinaldi2016}
    S. Albeverio, S. Fassari, and F. Rinaldi,
    The Hamiltonian of the harmonic oscillator with an attractive $\delta'$-interaction centred at the origin as approximated by the one with a triple of attractive $\delta$-interactions,
    \textit{J. Phys. A: Math. Theor.}, \textbf{49} (2016), no.~2, id.~025302.


\bibitem{FassariGadellaNietoRinaldi2024}
    S. Fassari, M. Gadella, L. M. Nieto, and F. Rinaldi,
    Analysis of a one-dimensional Hamiltonian with a singular double well consisting of two nonlocal $\delta'$ interactions,
    \textit{The European Physical Journal Plus}, \textbf{139} (2024), no.~2, 1--16.

\bibitem{AlbeverioKoshmanenkoKurasovNizhnik2002}
    S. Albeverio, V. Koshmanenko, P. Kurasov and L. Nizhnik,
  \textit{On approximations of rank one $\mathcal{H}_{-2}$-perturbations.}
    \textit{Proc. AMS} \textbf{131} (2003), no.~5, pp. 1443--1452.

\bibitem{KuzhelNizhnik2006}
    S. Kuzhel, L. Nizhnik,
    {Finite rank self-adjoint perturbations},
    \textit{Methods Funct. Anal. Topol.} \textbf{12} (2006), no.~3, 243--253. Available online:
    \url{http://mfat.imath.kiev.ua/article/?id=375}


\bibitem{AlbNizh2000}
    S. Albeverio and L. Niznik,
    Approximation of general zero-range potentials,
    \text{Ukra\"in. Math. Zh.} \textbf{52} (2000), no.~5, 582--589; translation in
    \textit{Ukrainian Math, J.} \textbf{52} (2000), no.~5, 664--672.


\bibitem{Fassari2018}
    S.~Fassari, M.~Gadella, M.~L. Glasser, and L.~M. Nieto,
    Spectroscopy of a one-dimensional V-shaped quantum well with a point impurity,
    \textit{Ann. Phys.} \textbf{389} (2018), 48--62.


\bibitem{GolovatyJPA2018}
    Y.~Golovaty,
    Two-parametric $\delta'$-interactions: approximation by Schrödinger operators with localized rank-two perturbations,
    \textit{J. Phys. A: Math. Theor.}, \textbf{51} (2018), no.~25, id.~255202.

\bibitem{GolovatyIEOT2018}
    Y.~Golovaty,
    Schr\"odinger operators with singular rank-two perturbations and point interactions,
    \textit{Int. Equat. and Oper. Theory}, \textbf{90} (2018), no.~5, id.~57.

\bibitem{CheonExner2004}
    T.~Cheon and P.~Exner,
    An approximation to $\delta'$ couplings on graphs.
    \textit{J. Phys. A: Math. General}, \textbf{37} (2004), no.~29, L329.

\bibitem{Manko2010}
    S.~S.~Man'ko,
    On $\delta'$-like potential scattering on star graphs,
    \textit{J. Phys. A: Math. Theor.}, \textbf{43} (2010),  no.~44, id.~445304.

\bibitem{ExnerManko2013}
    P.~Exner and S.~Manko,
    Approximations of quantum-graph vertex couplings by singularly scaled potentials,
    \textit{J. Phys. A: Math. Theor.}, \textbf{46} (2013),  no.~34, id.~345202.

\bibitem{GolovatyAHP2023}
    Y.~Golovaty,
    Quantum graphs: Coulomb-type potentials and exactly solvable models,
    \textit{Annales Henri Poincaré} \textbf{24} (2023), no.~8, pp. 2557--2585.

\bibitem{ReedSimonIV}
    M. Reed and B.  Simon,
    {\it  Methods of Modern Mathematical Physics, Vol. $4$: Analysis of Operators}.
    Academic Press, New York 1978.

\bibitem{BerezinShubin1991}
    F. A. Berezin and M. A. Shubin,
    \textit{The Schr\"{o}dinger equation},
    Kluwer Academic Publishers, 1991.

\bibitem{Simon2005}
    B.~Simon,
    \textit{Trace Ideals and Their Applications},
    2nd ed., American Mathematical Society, Providence, RI, 2005.

\bibitem{Klaus1977}
    M.~Klaus,
    On the bound state of Schr\"odinger operators in one dimension,
    \textit{Ann. Phys.} \textbf{108} (1977), no.~2, 288--300.



\bibitem{Kato}
    T.~Kato,
    \textit{Perturbation Theory for Linear Operators},
    Springer Science \& Business Media, 2013.

\bibitem{GolovatyFrontiers2019}
    Y. Golovaty,
    Some remarks on 1D Schr\"odinger operators with localized magnetic and electric potentials,
    \textit{Frontiers in Physics}, \textbf{7} (2019), id.~70.

\bibitem{Shkalikov2007}
    A. A. Shkalikov,
    Spectral analysis of the Regge problem,
    \textit{Journal of Mathematical Sciences}, \textbf{144} (2007), no.~4, 4292--4300.

\bibitem{Avakian87}
M.~P. Avakian, G.~S. Pogosyan, A.~N. Sissakian, and V.~M. Ter-Antonyan,
Spectroscopy of a singular linear oscillator,
\textit{Phys. Lett. A} \textbf{124} (1987), no.~4--5, 233--236.

\bibitem{FassariInglese94}
S.~Fassari and G.~Inglese,
On the spectrum of the harmonic oscillator with a $\delta$-type perturbation,
\textit{Helv. Phys. Acta} \textbf{67} (1994), no.~6, 650--659.


\bibitem{Simon1976}
    B.~Simon,
    The bound state of weakly coupled Schr\"odinger operators in one and two dimensions,
    \textit{Annals of Physics} \textbf{97} (1976), no.~2, 279--288.

\bibitem{PDEVinitiSpringer}
    V. F. Lazutkin,
    Quasiclassical asymptotics of eigenfunctions, in:
    \textit{Partial Differential Equations~V: Asymptotic Methods for Partial Differential Equations},
    Vol. 5, Springer Science \& Business Media, 1999.



\end{thebibliography}
\end{document}